\documentclass[11pt,reqno]{amsart}
\usepackage[margin=1in]{geometry}
\usepackage[T1]{fontenc}
\usepackage[utf8]{inputenc}
\usepackage{times}
\usepackage{amsmath,amssymb,amsthm,mathtools}
\usepackage{enumitem}
\usepackage[colorlinks=true,linkcolor=blue,citecolor=red,urlcolor=blue]{hyperref}

\hypersetup{pdftitle={Omega bounds for the Piltz divisor problem over number fields},
  pdfauthor={Nilmoni Karak and Kamalakshya Mahatab}}

\numberwithin{equation}{section}

\theoremstyle{plain}
\newtheorem{theorem}{Theorem}[section]
\newtheorem{proposition}[theorem]{Proposition}
\newtheorem{lemma}[theorem]{Lemma}
\newtheorem{corollary}[theorem]{Corollary}
\theoremstyle{remark}
\newtheorem{remark}[theorem]{Remark}

\DeclareMathOperator{\Res}{Res}
\DeclareMathOperator{\Gal}{Gal}
\DeclareMathOperator{\Frob}{Frob}
\DeclareMathOperator{\Real}{Re}
\DeclareMathOperator{\Imag}{Im}
\newcommand{\disc}{\mathfrak D}
\newcommand{\Q}{\mathbb{Q}}
\newcommand{\N}{\mathbb{N}}
\newcommand{\Z}{\mathbb{Z}}
\newcommand{\R}{\mathbb{R}}
\newcommand{\M}{\mathcal{M}}
\renewcommand{\le}{\leqslant}

\renewcommand{\ge}{\geqslant}

\title[Omega bounds for the Piltz divisor problem]{Omega bounds for the Piltz divisor problem }
\author{Nilmoni Karak}
	\address{ \scriptsize Nilmoni Karak, Department of Mathematics,\newline
		Indian Institute of Technology	Kharagpur,
		Kharagpur-721302, India}
	\email{nilmonikarak@gmail.com, nilmonimath@kgpian.iitkgp.ac.in}

	\author{Kamalakshya Mahatab}
	\address{ \scriptsize  Kamalakshya Mahatab, Department of Mathematics,\newline
		Indian Institute of Technology	Kharagpur,
		Kharagpur-721302, India}
	\email{accessing.infinity@gmail.com, kamalakshya@maths.iitkgp.ac.in}
	

\begin{document}
\begin{abstract}
We improve our earlier omega bounds for the error term in the Piltz divisor problem over number fields. Our proof uses a weighted extension of a counting theorem of Balasubramanian and Ramachandra and an extension of Lamzouri's kernel to apply the long resonance method. We also discuss applications of our results to lattice points in norm-form regions, invariant sublattices, abelian covers of curves, spectral counting, and representation growth of Heisenberg groups.
\end{abstract}
\maketitle

\section{Introduction}
\label{sec:introduction}
The Piltz divisor problem over a number field asks how accurately we
can count ordered products of integral ideals by their norms.
The Dirichlet divisor problem and the Gauss circle problem are
special cases. In this paper, we study large values of the error
in this count and the signs of these values.

Let $K$ be a number field of degree $m=r_1+2r_2$, where $r_1$ is the
number of real embeddings and $r_2$ is the number of pairs of complex
embeddings. For a positive integer $k$, define $d_K^{(k)}(n)$ by
\begin{equation}\label{eq:divisor-coefficients}
 \zeta_K(s)^k
 =\sum_{n=1}^{\infty}\frac{d_K^{(k)}(n)}{n^s},
 \qquad \Real s>1,
\end{equation}
where $\zeta_K$ is the Dedekind zeta function of $K$. Thus
$d_K^{(k)}(n)$ counts ordered $k$-tuples of nonzero integral ideals
whose product has norm $n$. The associated error term is
\begin{equation}\label{eq:error-term}
 \Delta_K^{(k)}(x)
 =\sum_{n\le x}d_K^{(k)}(n)
 -\Res_{s=1}\left(\zeta_K(s)^k\frac{x^s}{s}\right).
\end{equation}
For $K=\Q$ and $k=2$, this is the Dirichlet divisor error term.
For $K=\Q(i)$ and $k=1$, the quantity $4\Delta_K^{(1)}(x)$ is the
circle error term when the origin is omitted from the count.
Throughout the paper, $K$ and $k$ are fixed, $mk\ge2$, and
$\log_jx$ denotes the $j$-fold iterated natural logarithm.

We use the following omega notation. If $f$ is real-valued and
$g$ is positive, then $f=\Omega_+(g)$ means that $f(x)\ge cg(x)$
for some $c>0$ and arbitrarily large $x$.
We write $f=\Omega_-(g)$ when $-f=\Omega_+(g)$, and
$f=\Omega(g)$ when $|f|=\Omega_+(g)$.
Also, $f=\Omega_\pm(g)$ means that both signed bounds hold.

Szeg\H{o} and Walfisz~\cite{sw1,sw2} obtained early omega bounds
for~\eqref{eq:error-term}. Hafner~\cite{hafner83} improved these
bounds and also kept information about the sign. In particular,
his results give bounds of both signs when $mk\ge4$ or when $K$
is cubic and not totally real. Soundararajan~\cite{sound} later
obtained larger values in the classical divisor and circle problems.
His method also applies to the $k$-divisor problem, although these
larger bounds generally do not specify a sign.
We therefore consider both the size and the sign of large values.

To state the bounds, we first define the constants that come from
the splitting of rational primes in $K$. Let $L$ be the Galois
closure of $K$, and let $G=\Gal(L/\Q)$ act on the $m$ embeddings
of $K$ into $L$. For $g\in G$, let $v(g)$ be the number of fixed
points, and put
\begin{equation*}
 \delta_\nu
 =\frac{|\{g\in G:v(g)=\nu\}|}{|G|},
 \qquad 0\le\nu\le m.
\end{equation*}
By the Chebotarev density theorem, $\delta_\nu$ is the density of
rational primes unramified in $L$ that have exactly $\nu$ prime
ideals of residue degree one above them in $K$.
Also, Burnside's lemma gives $\sum_{\nu=1}^m\nu\delta_\nu=1$.
Now we put
\begin{equation}\label{eq:R-eta}
 \alpha=\frac{mk+1}{2mk},
 \qquad
 \eta=\frac{mk-1}{2mk},
 \qquad
 q=\frac1\alpha,
\end{equation}
and define
\begin{equation}\label{eq:splitting-exponent}
 \Theta=\sum_{\nu=1}^m\delta_\nu(k\nu)^q,
 \qquad
 \kappa=\alpha(\Theta-1).
\end{equation}
We also write
\begin{equation*}
 R=\#\{1\le\nu\le m:\delta_\nu>0\}.
\end{equation*}
Thus $R$ counts the possible positive numbers of degree-one prime
ideals above an unramified rational prime.

Using Soundararajan's method~\cite{sound}, Girstmair, K\"uhleitner, M\"uller and
Nowak~\cite{gkmn} proved that
\begin{equation}\label{eq:gkmn-bound}
 \Delta_K^{(k)}(x)=\Omega\!\left(
 (x\log x)^\eta(\log_2x)^\kappa
 (\log_3x)^{-\alpha R/2-\eta}\right).
\end{equation}
In our earlier work~\cite{km}, we used the resonance method to
improve this to
\begin{equation}\label{eq:km-bound}
 \Delta_K^{(k)}(x)=\Omega\!\left(
 (x\log x)^\eta(\log_2x)^\kappa
 (\log_3x)^{-\alpha R/2}\right).
\end{equation}
Both bounds hold with $\Omega_+$ when $kr_1\equiv3\pmod8$ and
with $\Omega_-$ when $kr_1\equiv7\pmod8$.

For the classical divisor and circle problems,
Mahatab~\cite{Mahatab(2026)} used resonance to improve the factor
involving $\log_3x$. Lamzouri~\cite{lamzouri} then removed this
factor and obtained a positive bound for the divisor problem and
a negative bound for the circle problem. His argument keeps the
phase in the Vorono\"i formula when applying resonance.
We combine this argument with a counting theorem that uses the
splitting of primes in $K$. The phase depends on $kr_1$ modulo
eight, so our conclusion also depends on this residue class.

\begin{theorem}\label{thm:main}
Let $K$ be a number field of degree $m$ with $r_1$ real embeddings,
and let $k$ be a positive integer with $mk\ge2$. Let $\alpha$, $\eta$
and $\kappa$ be defined by~\eqref{eq:R-eta}
and~\eqref{eq:splitting-exponent}.
\begin{enumerate}[label=(\roman*)]
\item If $kr_1\not\equiv1,5\pmod8$, then
\begin{equation}\label{eq:main-result}
 \Delta_K^{(k)}(x)=\Omega\!\left(
 (x\log x)^\eta(\log_2x)^\kappa\right).
\end{equation}
More precisely, the bound holds with $\Omega_+$ when
$kr_1\equiv2,3,4\pmod8$, and with $\Omega_-$ when
$kr_1\equiv0,6,7\pmod8$.
\item If $kr_1\equiv1,5\pmod8$, then
\begin{equation}\label{eq:exceptional-result}
 \Delta_K^{(k)}(x)=\Omega\!\left(
 (x\log x)^\eta(\log_2x)^\kappa
 (\log_3x)^{-\alpha/2}\right).
\end{equation}
\end{enumerate}
\end{theorem}
Part~(i) removes the factor involving $\log_3x$
from~\eqref{eq:km-bound} in six residue classes and gives the
signs stated above. Taking $K=\Q$, $k=2$, and $K=\Q(i)$, $k=1$,
we recover Lamzouri's bounds. Our improvement gives larger values
of one specified sign. Where the earlier results give bounds of
both signs, they still give a bound for the opposite sign.
In the two remaining classes, part~(ii) replaces the exponent
$-\alpha R/2$ by $-\alpha/2$. We therefore gain a factor
$(\log_3x)^{\alpha(R-1)/2}$ when $R>1$, and get the same order
of magnitude when $R=1$. For $K=\Q$, the exponent $-\alpha/2$
already appears in Mahatab~\cite[Theorem~3]{Mahatab(2026)}.

\begin{remark}\label{rem:galois}
We have $R=1$ if and only if $K/\Q$ is Galois.
If $K/\Q$ is Galois, only the identity fixes an embedding,
and it fixes all $m$ embeddings. Thus $R=1$.
Conversely, if $K/\Q$ is not Galois, then $\Gal(L/K)$ contains
a nonidentity element. This element fixes at least one embedding.
It cannot fix all of them, since their images generate $L$.
The identity fixes all $m$ embeddings, so we get $R\ge2$.
Thus part~(ii) improves the order of magnitude
in~\eqref{eq:km-bound} precisely when $K/\Q$ is not Galois.
\end{remark}

We now describe the main steps of the proof.
We first prove a uniform counting theorem for nonnegative
multiplicative functions whose values at primes are determined
by Frobenius conjugacy classes. For this, we adapt the
Selberg--Delange argument of Balasubramanian and Ramachandra~\cite{br},
which Lamzouri used in the classical setting.
With $a(n)=d_K^{(k)}(n)$, we obtain an asymptotic formula for the
sum of the largest $M$ values of $a(n)n^{-\alpha}$.
This sum has order $M^\eta(\log M)^\kappa$.
By adding a weight to the count, we also find large coefficients
whose indices lie in one interval $[U,eU]$.

The two parts of Theorem~\ref{thm:main} come from the phase
$\beta=(kr_1-3)\pi/4$ in the smoothed Vorono\"i formula.
When $\cos\beta\ne0$, Lamzouri's kernel gives the signed bound
in part~(i). When $kr_1\equiv1,5\pmod8$, the oscillatory sum
is a sine sum, up to a sign. We convolve it with an odd Gaussian
kernel and get a cosine sum with nonnegative coefficients.
We then apply resonance using the indices selected in $[U,eU]$.
This use of an odd kernel follows the approach of Bondarenko
and Seip~\cite[Section~5]{BS}. Related convolution arguments
that keep track of signs appear in Chirre and Mahatab~\cite{CM}.
Here the convolution gives a bound for the absolute value of
the original sine sum.

To select the indices, we count those with
$b_n=na(n)^{-q}\le X$, where $a(n)>0$.
The weighted counting theorem shows that the logarithms of a
positive proportion of these indices lie in an interval of
length $O(\sqrt{\log_2X})$.
We divide this interval into intervals of length at most one
and choose one containing enough indices.
This gives the factor $(\log_3x)^{-\alpha/2}$ in part~(ii),
and its exponent does not depend on $R$.
\subsection{Additional Applications}
Our results also give omega bounds for the error terms in several
counting problems. In Section~\ref{sec:applications}, we discuss
the following applications.
\begin{itemize}
\item \emph{Quadratic forms and norm-form regions.}
We count representations by positive-definite binary quadratic
forms, summed over the classes of a fixed fundamental discriminant.
We also count lattice points in explicit cubic and quartic
norm-form regions.

\item \emph{Invariant sublattices and colourings.}
We consider sublattices preserved by a signed cyclic shift and
chains of such sublattices. We also consider Bravais colourings
for cyclotomic rings of class number one.

\item \emph{Covers of curves.}
For a family of hyperelliptic curves, we count connected unramified
abelian Galois covers to which a specified automorphism lifts.
Their degrees equal the norms of the corresponding ideals, so
the theorem gives bounds for the counting error.

\item \emph{Spectral counting.}
We count eigenvalues of tensor powers of the Laplacian on the
square torus. By restricting each factor to mean-zero functions,
we can write these counts as divisor sums over the Gaussian integers.

\item \emph{Representation growth.}
For Heisenberg groups over rings of integers and their direct
powers, we count irreducible complex representations up to
twists by one-dimensional characters.
\end{itemize}
Most of these counts are given by exact identities with ideal
counts. For representation growth, we use a further argument
based on positivity. This gives the required bounds and preserves
the sign whenever the main theorem gives a signed bound.
\section*{Acknowledgements}
Karak is supported by the Prime Minister's Research Fellowship (PMRF),
Government of India (PMRF ID: 2403449). Mahatab is supported by the
ARG-MATRICS program, Government of India
(Grant No.\ ANRF/\allowbreak ARGM/\allowbreak 2025/\allowbreak
002540/\allowbreak MTR).

\section{A counting theorem}
\label{sec:counting}

The coefficients used in resonance are best ordered by size, rather
than by their indices. The following theorem counts the indices
for which $a(n)n^{-1/q}$ is at least a given threshold. The additional weight $n^z$
will allow us to locate many of these indices in one multiplicative
interval.

\begin{theorem}\label{thm:count}
Let $a:\N\to[0,\infty)$ be multiplicative, with $a(1)=1$. Assume that:
\begin{enumerate}[label=(\roman*)]
\item $a(n)\ll_\varepsilon n^\varepsilon$ for every $\varepsilon>0$;
\item $a(p^j)\le C(j+1)^D$ for all primes $p$ and integers $j\ge1$,
where $C,D$ are positive constants;
\item there are a finite Galois extension $L/\Q$ with group $G$ and
a nonnegative class function $\nu$, not identically zero, such that
$a(p)=\nu(\Frob_p)$ for all but finitely many primes $p$.
\end{enumerate}
Fix $q>0$ and $0<z_0<1/4$. Let $b_n=na(n)^{-q}$ when $a(n)>0$.
For real $z\in[-z_0,z_0]$, define
\begin{equation*}
 A_z(x)=\sum_{\substack{a(n)>0\\b_n\le x}}n^z
\end{equation*}
and
\begin{equation*}
 \theta_q(z)=\frac1{|G|}
 \sum_{\substack{g\in G\\\nu(g)>0}}\nu(g)^{q(1+z)}.
\end{equation*}
Then, uniformly for $z\in[-z_0,z_0]$ as $x\to\infty$,
\begin{equation}\label{eq:weighted-count}
 A_z(x)\sim C(z)x^{1+z}(\log x)^{\theta_q(z)-1},
\end{equation}
where $C(z)=H_z(1)/((1+z)\Gamma(\theta_q(z)))>0$.
The function $H_z$ is analytic near $1$, and $C(z)$ is continuous in
$z$. The functions $H_z$ are defined in the proof. In particular, if
$H=H_0$ and $\theta_q=\theta_q(0)$, then
\begin{equation}\label{eq:count}
 \#\{n:a(n)>0,\ na(n)^{-q}\le x\}
 \sim\frac{H(1)}{\Gamma(\theta_q)}x(\log x)^{\theta_q-1}.
\end{equation}
\end{theorem}

\begin{proof}
Choose $0<\varepsilon<1/q$. Assumption~(i) gives
$b_n\gg_\varepsilon n^{1-q\varepsilon}$ whenever $a(n)>0$.
Thus $b_n\to\infty$ as $n\to\infty$ through these integers, so
$A_z(x)$ is finite.
We begin with the generalized Dirichlet series
\begin{equation}\label{eq:Dirichlet}
 \mathcal F_z(s)=\sum_{a(n)>0}n^zb_n^{-s},
\end{equation}
and define
\begin{equation*}
 \mathcal G_z(w)=\mathcal F_z(w+z)
 =\sum_{a(n)>0}a(n)^{q(w+z)}n^{-w}.
\end{equation*}
By~(i), this series converges absolutely for $\Real w>1$.
The convergence is uniform in $z$ on every strip
$1<\sigma_1\le\Real w\le\sigma_2$.
Since $a$ is multiplicative, we have
$\mathcal G_z(w)=\prod_pF_{p,z}(w)$, where
\begin{equation*}
 F_{p,z}(w)=1+\sum_{\substack{j\ge1\\a(p^j)>0}}
 a(p^j)^{q(w+z)}p^{-jw}.
\end{equation*}
Here we define powers of positive real numbers using their real logarithms.
Let $S$ contain the primes ramified in $L$ and the exceptions in~(iii).
For $g\in G$, we put
\begin{equation*}
 h_{z,w}(g)=
 \begin{cases}
  \exp(q(w+z)\log\nu(g)),&\nu(g)>0,\\
  0,&\nu(g)=0.
 \end{cases}
\end{equation*}
Now let us fix a strip $1/2<\sigma_1\le\Real w\le\sigma_2$.
Since $\Real w+z>1/4$, values $0<a(p^j)\le1$ contribute at most
$p^{-j\Real w}$. For the other values, condition~(ii) gives a fixed
power of $j+1$ as an upper bound. We therefore have
\begin{equation*}
 \sum_{\substack{j\ge2\\a(p^j)>0}}
 |a(p^j)^{q(w+z)}p^{-jw}|\ll p^{-2\Real w},
\end{equation*}
uniformly in $z$ and $\Imag w$. Hence, for $p\notin S$, we get
\begin{equation}\label{eq:simplified-series}
 F_{p,z}(w)=1+h_{z,w}(\Frob_p)p^{-w}
 +O(p^{-2\Real w}).
\end{equation}
Suppose $\widehat G$ denote the set of all irreducible complex characters of $G$. Since $h_s$ is a class function, the orthogonality formula gives $h_{z,w}=\sum_{\chi\in\widehat G}c_\chi(z,w)\chi$, where
\begin{equation*}
c_\chi(z,w)=\frac1{|G|}\sum_{g\in G}
 h_{z,w}(g)\overline{\chi(g)}.
\end{equation*}
Observe that the coefficients $c_{\chi}$ are analytic in $w$, continuous in $z$, and bounded
on the strips above. The coefficient of the trivial character is
\begin{equation*}
 Z_z(w)=\frac1{|G|}\sum_{\nu(g)>0}\nu(g)^{q(w+z)}.
\end{equation*}
so, $Z_z(1)=\theta_q(z)$.
Let $L(w,\chi)=\prod_pL_p(w,\chi)$ be the Artin $L$-function
attached to $\chi$. For an unramified prime $p$, we have
\begin{equation*}
 \log L_p(w,\chi)=\chi(\Frob_p)p^{-w}
 +O(p^{-2\Real w}).
\end{equation*} 
By~\eqref{eq:simplified-series}, we have
$E_{p,z}(w)=1+O(p^{-2\Real w})$ for $p\notin S$.
For the finitely many primes in $S$, the same definition gives
analytic functions for $\Real w>1/2$, bounded uniformly in $z$
on every strip $1/2<\sigma_1\le\Real w\le\sigma_2$.
It follows that $ \mathcal E_z(w)=\prod_pE_{p,z}(w)$
converges uniformly on compact subsets of $\Real w>1/2$.
It is analytic there and is bounded uniformly in $z$ on each such strip. As the trivial Artin factor is Riemann zeta function $\zeta(s)$, for $\Real w>1$, we obtain
\begin{equation}\label{eq:zeta-factor}
 \mathcal G_z(w)=\zeta(w)^{Z_z(w)}H_z(w),
\end{equation}
where
\begin{equation*}
 H_z(w)=\mathcal E_z(w)
 \exp\left(\sum_{\chi\ne1}c_\chi(z,w)\log L(w,\chi)\right).
\end{equation*}

We now use the continuation described in
\cite[pp.~196--197, (17)--(20)]{gkmn}.
By Brauer induction, each Artin $L$-function is a finite product
of integer powers of Hecke $L$-functions. For a nontrivial irreducible
character $\chi$, the poles at $1$ cancel because their total
multiplicity is $\langle\chi,1\rangle_G=0$.
The nontrivial Hecke factors are nonzero at $1$.
Since the field is fixed, we can therefore choose $c>0$ so that
all nontrivial Artin factors are holomorphic and nonzero in
\begin{equation}\label{eq:region}
 \mathcal D_c=\left\{\sigma+it:
 \sigma>1-\frac{c}{\log(|t|+3)}\right\}.
\end{equation}
We take $c$ small enough to exclude any exceptional real zero and
to ensure that $\mathcal D_c$ lies in $\Real w>1/2$.
We then continue the logarithms from their Euler product definitions
in $\Real w>1$. For the power of $\zeta(w)$, we remove the real
segment in $\mathcal D_c$ to the left of $1$, together with $1$ itself. After decreasing $c$ if needed, the logarithmic estimate
\cite[p.~196, (18)]{gkmn} gives
\begin{equation*}
 |\log L(\sigma+it,\chi)|\ll\log\log(|t|+3)
 \quad \text{for} \quad |t|\ge1
\end{equation*}
in $\mathcal D_c$, including the trivial character.
Since the coefficients $c_\chi(z,w)$ and $\mathcal E_z(w)$ are
bounded on each fixed strip, this gives
\begin{equation}\label{eq:growth}
 |\mathcal G_z(\sigma+it)|\ll(|t|+3)^{C_1}
\end{equation}
for $|t|\ge1$ in $\mathcal D_c$ on that strip, with $C_1$
independent of $z$. Also, the formula for $H_z$ shows that it is
analytic in a neighbourhood of $1$ and continuous in $z$ there.
Uniform convergence of the products gives uniform bounds on a
smaller closed neighbourhood, and Cauchy's estimates then give uniform
bounds for the derivatives with respect to $w$.

For real $z$, we have $F_{p,z}(1)>0$, so every $E_{p,z}(1)$ is
nonzero. Since $\sum_{p\notin S}|E_{p,z}(1)-1|$ converges,
$\mathcal E_z(1)$ is nonzero, and hence $H_z(1)\ne0$.
By continuity and compactness, $|H_z(1)|$ is bounded away from zero.
We also have $Z_z(1+v)=\theta_q(z)+O(|v|)$ and
$v\zeta(1+v)=1+O(|v|)$. Using these estimates
in~\eqref{eq:zeta-factor}, we get
\begin{equation}\label{eq:singularity}
 \mathcal G_z(1+v)=H_z(1)v^{-\theta_q(z)}
 \left(1+O\bigl(|v|(1+|\log v|)\bigr)\right),
\end{equation}
uniformly in $z$, for small $v$ with $-\pi<\arg v<\pi$.
Here the factor $\log v$ comes from the variation of $Z_z(1+v)$.
As $v\to0^+$, the series for $\mathcal G_z(1+v)$ is positive,
so $H_z(1)>0$. Also, $\theta_q(z)$ is continuous and positive.
Thus both $H_z(1)$ and $\theta_q(z)$ have positive lower bounds
and finite upper bounds on $[-z_0,z_0]$.

We take $\rho=1+z$, choose an integer $J>C_1+3$ and define
\begin{equation*}
 A_{z,J}(x)=\frac1{J!}\sum_{\substack{a(n)>0\\b_n\le x}}
 n^z\left(\log\frac{x}{b_n}\right)^J.
\end{equation*}
For $b=\rho+1/\log x$, the Mellin inversion formula gives
\begin{equation}\label{eq:mellin}
 A_{z,J}(x)=\frac1{2\pi i}\int_{b-i\infty}^{b+i\infty}
 \mathcal F_z(s)\frac{x^s}{s^{J+1}}\,ds.
\end{equation}
The above formula holds for positive real $b_n$, and they need not be integers.
We may interchange the sum and integral by absolute convergence.
Now we put $\ell=\log x$ and $T=\exp(\sqrt\ell)$.
By positivity and~\eqref{eq:singularity}, we have
\begin{equation*}
 |\mathcal F_z(b+it)|\le\mathcal F_z(b)
 \ll\ell^{\theta_q(z)}.
\end{equation*}
Thus, for $|t|>T$, the part of~\eqref{eq:mellin}  is
$O(x^\rho\ell^{\theta_q(z)}T^{-J})$.
Let $0<c_1<c$ and $d=c_1/\log(T+3)$.
We move the truncated line to $\Real s=\rho-d$, with a Hankel
contour around $s=\rho$. We take this contour along the two sides
of the segment from $\rho-d$ to $\rho-\ell^{-1}$ and around
the circle $|s-\rho|=\ell^{-1}$.
The new path lies in the region of continuation, and it does not
cross $s=0$ when $x$ is large.
By~\eqref{eq:growth}, the two horizontal integrals contribute
$O(x^\rho T^{C_1-J-1})$.
On the new vertical line with $|t|\ge1$, the same bound and
$J>C_1$ give $O(x^{\rho-d})$.
Near $s=\rho$, equation~\eqref{eq:singularity} bounds the remaining
integral by $x^{\rho-d}$ times a fixed power of $d^{-1}$.
The part with bounded $t$ away from $\rho$ is uniformly bounded
after removing the factor $x^{\rho-d}$.
Since $d\ell\asymp\sqrt\ell$, all these terms are
$O(x^\rho e^{-c_2\sqrt\ell})$ for some $c_2>0$.

On the Hankel contour, let $s=\rho+v$.
By~\eqref{eq:singularity}, we have
\begin{equation*}
 \mathcal F_z(\rho+v)(\rho+v)^{-J-1}
 =\frac{H_z(1)}{\rho^{J+1}}v^{-\theta_q(z)}
 +O\bigl(|v|^{1-\theta_q(z)}(1+|\log v|)\bigr).
\end{equation*}
We now put $u=\ell v$ and use Hankel's formula
\begin{equation*}
 \frac1{2\pi i}\int_{\mathcal H}e^uu^{-\theta_q(z)}\,du
 =\frac1{\Gamma(\theta_q(z))}.
\end{equation*}
Here $\mathcal H$ runs from $-\infty$ below the negative real axis,
around the unit circle, and back to $-\infty$ above that axis.
Extending the two rays from $-d\ell$ to $-\infty$ gives an
exponentially small error. On the circle and the two finite rays,
the error term contributes
$O(x^\rho\ell^{\theta_q(z)-2}(1+\log\ell))$.
We therefore obtain
\begin{equation}\label{eq:smoothed}
 A_{z,J}(x)=\frac{H_z(1)}{\rho^{J+1}\Gamma(\theta_q(z))}
 x^\rho(\log x)^{\theta_q(z)-1}
 \left(1+O\left(\frac{\log_2x}{\log x}\right)\right),
\end{equation}
uniformly in $z$.

We now remove the smoothing. We define $A_{z,j}$ in the same way
for $1\le j\le J$ with $A_{z,0}=A_z$, and take
$U_{z,j}(v)=A_{z,j}(e^v)$.
The weights are positive, so $U_{z,j-1}$ is nondecreasing.
Also,
\begin{equation*}
 U_{z,j}(v+h)-U_{z,j}(v)
 =\int_v^{v+h}U_{z,j-1}(u)\,du.
\end{equation*}
It follows that, for $h>0$,
\begin{equation}\label{eq:unsmoothing}
 \frac{U_{z,j}(v)-U_{z,j}(v-h)}h
 \le U_{z,j-1}(v)
 \le\frac{U_{z,j}(v+h)-U_{z,j}(v)}h.
\end{equation}
Suppose that
$U_{z,j}(v)\sim B_j(z)e^{\rho v}v^{\theta_q(z)-1}$ uniformly in $z$.
After division by $B_j(z)e^{\rho v}v^{\theta_q(z)-1}$, the two
bounds in~\eqref{eq:unsmoothing} tend uniformly to
$(1-e^{-\rho h})/h$ and $(e^{\rho h}-1)/h$.
We first let $v\to\infty$ and then let $h\to0$.
Since $\rho$ stays in $[1-z_0,1+z_0]$, both limits tend uniformly
to $\rho$. Thus the same asymptotic holds for $U_{z,j-1}$ with
$B_{j-1}(z)=\rho B_j(z)$.
Applying this step $J$ times to~\eqref{eq:smoothed} gives
\eqref{eq:weighted-count}. Taking $z=0$ gives~\eqref{eq:count}.
\end{proof}

\subsection{The largest divisor coefficients}
For the rest of the paper, we consider $a(n)=d_K^{(k)}(n)$. At a rational prime
$p$, the Euler factor gives
\begin{equation}\label{eq:ideals}
 \sum_{j\ge0}a(p^j)z^j
 =\prod_{\mathfrak p\mid p}(1-z^{f_{\mathfrak p}})^{-k},
\end{equation}
where $f_{\mathfrak p}$ is the residue degree. We expand each factor
$(1-z^{f_{\mathfrak p}})^{-k}$ as a product of $k$ geometric series. Here each
 contribution to the coefficient of $z^j$ then corresponds to a touple of
nonnegative integers $(b_i)$ satisfying $\sum_i f_ib_i=j$. Note that the map
$(b_i)\mapsto(f_ib_i)$ is injective, and we can add zero entries
to obtain a tuple of length $mk$ whose entries sum to $j$.
Thus, also for ramified primes, we have
\begin{equation*}
 a(p^j)\le\binom{j+mk-1}{mk-1}\ll(j+1)^{mk-1}.
\end{equation*}
By multiplicativity, it follows that $a(n)\le d_{mk}(n)\ll_\varepsilon n^\varepsilon$.
For an unramified prime in the Galois closure of $K$, the coefficient of $z$
in~\eqref{eq:ideals} is $a(p)=kv(\Frob_p)$. Therefore, we apply Theorem~\ref{thm:count} with the class function
$kv$ and with $q$ from in~\eqref{eq:R-eta}. This gives
\begin{equation}\label{eq:applicationcount}
 A_0(x)\sim Cx(\log x)^{\Theta-1},
\end{equation}
where $C=H(1)/\Gamma(\Theta)>0$.

We also have $\Theta>1$. Indeed, $q>1$, and Burnside's lemma gives
$|G|^{-1}\sum_{g\in G}v(g)=1$.
By convexity, we get $|G|^{-1}\sum_{g\in G}v(g)^q\ge1$.
Hence $\Theta\ge k^q>1$ when $k\ge2$.
If $k=1$, then $m>1$. The identity has $m$ fixed points, whereas
the average number of fixed points is one, so $v$ is not constant.
Strict convexity then gives $\Theta>1$ in this case as well.

\begin{proposition}\label{prop:largest}
Let $S_K(M)$ be the sum of the largest $M$ positive values of
$a(n)n^{-\alpha}$. As $M\to\infty$, we have
\begin{equation}\label{eq:largest}
 S_K(M)\sim\frac{C^\alpha}{\eta}M^\eta(\log M)^\kappa.
\end{equation}
\end{proposition}

\begin{proof}
We define
\begin{equation*}
 N_K(y)=\#\{n:a(n)>0,\ n^\alpha/a(n)\le y\}.
\end{equation*}
Since $q=1/\alpha$, we have $N_K(y)=A_0(y^q)$. Thus the asymptotic formula
\eqref{eq:applicationcount} implies,
\begin{equation}\label{eq:asmp_N_K}
 N_K(y)\sim Cq^{\Theta-1}y^q(\log y)^{\Theta-1}.
\end{equation}
By partial summation formula, we then get
\begin{equation}\label{eq:partial-sum}
 \begin{aligned}
 B_K(y)&:=\sum_{\substack{a(n)>0\\n^\alpha/a(n)\le y}}
               a(n)n^{-\alpha}\\
 &=\frac{N_K(y)}y+\int_0^y\frac{N_K(t)}{t^2}\,dt\\
 &\sim\frac{Cq^{\Theta-1}}{\eta}y^{q-1}(\log y)^{\Theta-1},
 \end{aligned}
\end{equation}
where we used $q/(q-1)=1/\eta$. Also, by the divisor bound, $a(n) n^{-\alpha}$ tends to zero, so we can  therefore arrange its positive values in decreasing,
with repetitions, and denote their indices by $n_1,n_2,\ldots$.
We now let $y_M=n_M^\alpha/a(n_M)$. For $0<\rho<1$, we have $N_K((1-\rho)y_M)<M\le N_K(y_M)$.
Using~\eqref{eq:asmp_N_K} and then taking $\rho\to0$, we get
$M\sim N_K(y_M)$, which implies that
\begin{equation}\label{eq:asymp_y}
 y_M\sim C^{-\alpha}M^\alpha(\log M)^{-\kappa}.
\end{equation}
Finally, $B_K((1-\rho)y_M)\le S_K(M)\le B_K(y_M)$ and 
using~\eqref{eq:partial-sum} and letting $\rho\to0$ once more, we obtain
$S_K(M)\sim M/(\eta y_M)$. Combining this with~\eqref{eq:asymp_y} completes the proof of the proposition.
\end{proof}

\subsection{Selection in a multiplicative interval}
\label{subsec:band}
For the exceptional phases, we need large coefficients whose indices
lie in one interval $[U,eU]$. We now use the weight $n^z$ to find
such an interval. We retain $b_n=na(n)^{-q}$ when $a(n)>0$, and let
\begin{equation*}
 \Psi(z)=\sum_{\nu=1}^m\delta_\nu(k\nu)^{q(1+z)}.
\end{equation*}
By Theorem~\ref{thm:count}, we have
\begin{equation}\label{eq:weighted-application}
 A_z(X)\sim C(z)X^{1+z}(\log X)^{\Psi(z)-1}
\end{equation}
uniformly for real $z$ in a neighbourhood of zero.
Also, $\Psi(0)=\Theta$ and $C(0)=C$.

\begin{proposition}\label{prop:band}
For every sufficiently large $M$, there exist $U>0$ and a set
$\M\subset[U,eU]\cap\N$ of $M$ indices such that
\begin{equation}\label{eq:band-mass}
 \sum_{n\in\M}a(n)n^{-\alpha}
 \gg M^\eta(\log M)^\kappa(\log_2M)^{-\alpha/2}.
\end{equation}
Moreover, every selected index satisfies $n\le M(\log M)^{C_2}$ for some $C_2>0$.
\end{proposition}

\begin{proof}
Define the counting function
\begin{equation*}
 \mathcal N_X(u)=\#\{n\le u:a(n)>0,\ b_n\le X\}.
\end{equation*}
Suppose $\ell=\log_2X$, $\mu=\Psi'(0)$ and
$V_X=X(\log X)^\mu$. By~\eqref{eq:weighted-application}, the Stieltjes integral gives
\begin{equation}\label{eq:weighted-integral}
 \frac1{A_0(X)}\int_{1^-}^{\infty}
 \left(\frac{u}{V_X}\right)^z\,d\mathcal N_X(u)
 =\frac{C(z)}{C(0)}
 \exp\bigl((\Psi(z)-\Theta-\mu z)\ell\bigr)(1+o(1)),
\end{equation}
uniformly for real $z$ near zero. Therefore we may  take $z=\ell^{-1/2}$ and $z=-\ell^{-1/2}$. For either choice, Taylor's formula gives
\begin{equation*}
 (\Psi(z)-\Theta-\mu z)\ell
 =\tfrac12\Psi''(0)+O(\ell^{-1/2}).
\end{equation*}
Also $C(z)/C(0)\to1$ by continuity. Thus each integral in~\eqref{eq:weighted-integral}, before division
by $A_0(X)$, is $O(A_0(X))$.
For any $B>0$, it follows that
\begin{equation*}
 \begin{aligned}
 &\int_{\substack{u\ge1\\|\log(u/V_X)|>B\sqrt\ell}}
            d\mathcal N_X(u)\\
 &\quad\le e^{-B}\sum_{\epsilon\in\{-1,1\}}
       \int_{1^-}^{\infty}
       \left(\frac{u}{V_X}\right)^{\epsilon/\sqrt\ell}
       d\mathcal N_X(u)
 \ll e^{-B}A_0(X).
 \end{aligned}
\end{equation*}
We choose $B$ large enough that this is at most $A_0(X)/2$. At least half the integers counted by $A_0(X)$ then satisfy
\begin{equation}\label{eq:concentration}
 |\log(n/X)-\mu\log_2X|\le B\sqrt{\log_2X}.
\end{equation}
Now take
\begin{equation*}
 X=C_3M(\log M)^{1-\Theta}(\log_2M)^{1/2},
\end{equation*}
where $C_3$ is a sufficiently large constant. Since $\log X\sim\log M$ and $\log_2X\sim\log_2M$, we get
$A_0(X)\asymp C_3M\sqrt{\log_2M}$. Now the interval for $\log n$ in~\eqref{eq:concentration} has length
$2B\sqrt{\log_2X}$.
We partition it into disjoint intervals
of length at most one. Thus the number of intervals is at most
$2B\sqrt{\log_2X}+1$. By choosing $C_3$ sufficiently large, we ensure that one interval
contains at least $M$ of the integers under consideration.
We choose $M$ of them and call this set $\M$.
Then $\M\subset[U,eU]$ for some $U>0$, and each selected integer
satisfies
$a(n)n^{-\alpha}=b_n^{-\alpha}\ge X^{-\alpha}$. The sum $\sum_{n\in \M}a(n)n^{-\alpha}$ is therefore at least $MX^{-\alpha}$.
Using $\eta=1-\alpha$ and $\kappa=\alpha(\Theta-1)$,
we obtain~\eqref{eq:band-mass}.
Finally,~\eqref{eq:concentration} gives
\begin{equation*}
 n\le X\exp\bigl(\mu\log_2X+B\sqrt{\log_2X}\bigr)
 \le M(\log M)^{C_2}
\end{equation*}
for a sufficiently large $C_2>0$ and all sufficiently large $M$.

\end{proof}

\section{Resonance and convolution}
We first give the form of Lamzouri's argument
\cite[Theorem~1.2 and Proposition~2.1]{lamzouri} that we will use.
The same kernel works for all three phases below.
We include the proof to give the constant and the error term explicitly.

\begin{lemma}\label{lem:resonance}
Suppose that $f(n)\ge0$ with $\sum_nf(n)<\infty$ and  $\lambda_n>0$.
We write $F(t)=\sum_nf(n)e^{2\pi i\lambda_nt}$.
Let $\beta_0\in\{-\pi/4,0,\pi/4\}$, and let $\M$ be a set of $M$
indices. Then, for $3<Y<X$, we have
\begin{equation}\label{eq:resonance}
 \max_{Y\le t\le X}\Real\left(e^{i\beta_0}F(t)\right)
 \ge\frac1{12}\sum_{n\in\M}f(n)
 +O\left(F(0)2^M\left(\frac{Y\log X}{X}\right)^{1/4}\right).
\end{equation}
\end{lemma}

\begin{proof}
For $0<r<1$, we define
\begin{equation*}
 R(t)=\prod_{n\in\M}\left(1-re^{2\pi i\lambda_nt}\right)^{-1}.
\end{equation*}
Let $\mathcal A$ be the set of nonnegative integer combinations of the
selected frequencies. By expanding the product, we get
\begin{equation} \label{eq:R(t)}
 R(t)=\sum_{v\in\mathcal A}b_r(v)e^{2\pi ivt},
\end{equation}
where
\begin{equation*}
 b_r(v)=\sum_{\substack{c:\M\to\Z_{\ge0}\\
                 \sum_{n\in\M}c(n)\lambda_n=v}}
 r^{\sum_{n\in\M}c(n)}.
\end{equation*}
Since the selected frequencies are positive and $\M$ is finite,
each coefficient is a finite sum. Also, $\sum_{v\in\mathcal A}b_r(v)=(1-r)^{-M}$, so the series $\sum_{v\in\mathcal A}b_r(v)e^{2\pi ivt}$ converges absolutely and uniformly in $t$.
Adding one copy of $\lambda_n$ to each representation of $v$ gives
\begin{equation}\label{eq:br}
 b_r(v+\lambda_n)\ge rb_r(v)
\end{equation}
for $n\in\M$. We now consider the nonnegative kernel
\begin{equation*}
 K(u)=\begin{cases}
  \tfrac{u^{-3/4}e^{-u}}{\Gamma(1/4)},&u>0,\\
  0,&u\le0.
 \end{cases}
\end{equation*}
We use the convention
$\widehat K(\xi)=\int_{\R}K(u)e^{-2\pi i\xi u}\,du$.
Here $\widehat K(\xi)=(1+2\pi i\xi)^{-1/4}$.
Its argument lies in $(-\pi/8,\pi/8)$, so
\begin{equation}\label{eq:sector}
 \Real(e^{i\beta_0}\widehat K(\xi))
 \ge\frac14\Real\widehat K(\xi)\ge0.
\end{equation}
Indeed, for $\beta_0=\pm\pi/4$ the ratio of the two real parts is at
least $\sqrt2-1>1/4$, and for $\beta_0=0$ it is one.
We put $\tau=X/(2\log X)$, and define
 \begin{equation*}
I_1=\int_0^\infty|R(t)|^2K(t/\tau)\,dt
 \end{equation*}
and 
\begin{equation*}
I_2=\int_0^\infty\Real(e^{i\beta_0}F(t))|R(t)|^2K(t/\tau)\,dt.
\end{equation*}
We may expand the series \eqref{eq:R(t)} and integrate term by term because they
converge absolutely and $K$ is integrable.
Using~\eqref{eq:sector}, we then get
\begin{equation*}
 I_2\ge\frac\tau4\sum_nf(n)
 \sum_{u,v\in\mathcal A}b_r(u)b_r(v)
 \Real\widehat K((v-u-\lambda_n)\tau).
\end{equation*}
Observe that all terms on the right-hand side are nonnegative. We keep those with $n\in\M$
and $v=w+\lambda_n$, where $w\in\mathcal A$.
By~\eqref{eq:br}, this gives
\begin{equation}\label{eq:IJ}
 \begin{aligned}
 I_2&\ge\frac{r\tau}4\sum_{n\in\M}f(n)
       \sum_{u,w\in\mathcal A}b_r(u)b_r(w)
       \Real\widehat K((w-u)\tau)\\
    &=\frac r4\sum_{n\in\M}f(n)I_1.
 \end{aligned}
\end{equation}
We also keep the diagonal terms in the expansion of $I_1$.
\begin{equation*}
 I_1\ge\tau\sum_vb_r(v)^2
 \ge\tau\sum_vb_{r^2}(v)
 =\tau(1-r^2)^{-M}.
\end{equation*}

Let $V$ denote the maximum in~\eqref{eq:resonance}, and let $Q$ be the
mass of $|R(t)|^2K(t/\tau)\,dt$ outside $[Y,X]$. Since both $|V|$ and
$|\Real(e^{i\beta_0}F(t))|$ are at most $F(0)$,
\begin{equation*}
 I_2\le V(I_1-Q)+F(0)Q\le VI_1+2F(0)Q.
\end{equation*}
The bound $|R(t)|\le(1-r)^{-M}$ and the kernel estimates give
\begin{equation*}
 Q\ll(1-r)^{-2M}\tau
 \left((Y/\tau)^{1/4}+e^{-X/\tau}\right).
\end{equation*}
Since $X/\tau=2\log X$ and $e^{-X/\tau}=X^{-2}$, it follows that
\begin{equation*}
 \frac Q{I_1}\ll
 \left(\frac{1+r}{1-r}\right)^M
 \left(\frac{Y\log X}{X}\right)^{1/4}.
\end{equation*}
Combining this with~\eqref{eq:IJ}, we get
\begin{equation*}
 V\ge\frac r4\sum_{n\in\M}f(n)-2F(0)\frac Q{I_1}.
\end{equation*}
We now take $r=1/3$. Then $r/4=1/12$ and
$(1+r)/(1-r)=2$, which completes the proof of the lemma.
\end{proof}

\subsection{The smoothed Vorono\"i formula}
Let $\disc$ be the absolute discriminant of $K$.
For $N,t\ge2$, we define
\begin{equation}\label{eq:gaussian}
 \mathcal V_N(t)=\frac{N^{\frac{1}{mk}}}{\sqrt\pi}
 \int_{-\infty}^{\infty}\Delta_K^{(k)}(t^{mk}e^{u/t})e^{-u^2N^{\frac{2}{mk}}}\,du.
\end{equation}
We also consider
\begin{equation}\label{def:f(n)}
 f_N(n)=a(n)n^{-\alpha}
 \exp\left(-\pi^2\left(\frac n{\disc^kN}\right)^{\frac{2}{mk}}\right),
\end{equation}
\begin{equation}\label{eq:phase}
 \begin{aligned}
  \lambda_n&=mk\left(\frac n{\disc^k}\right)^{\frac{1}{mk}} \quad \text{and}\quad 
  \beta&=\frac\pi4(kr_1-3).
 \end{aligned}
\end{equation}
The divisor bound shows that $\sum_nf_N(n)<\infty$ for every $N$.
We now use the smoothed formula from
\cite[Proposition~2.2 and its proof]{km}, with the error estimate
given at the end of that proof.

\begin{proposition}\label{prop:voronoi}
For every $\epsilon>0$, we have
\begin{equation}\label{eq:voronoi}
 \begin{aligned}
 \mathcal V_N(t)
 &=\frac{\disc^{1/(2m)}}{\pi\sqrt{mk}}t^{(mk-1)/2}
   \sum_{n=1}^\infty f_N(n)\cos(2\pi\lambda_nt+\beta)
   +O\left(t^{\frac{mk}{2}-\frac35}N\right),
 \end{aligned}
\end{equation}
uniformly for sufficiently large $t$ and $2\le N\le(\log t)^3$.
\end{proposition}

We next use the positivity of the Gaussian to pass from
$\mathcal V_N(t)$ to a value of $\Delta_K^{(k)}$ nearby.
\begin{lemma}\label{lem:transfer}
For $N\ge2$, sufficiently large $t$, and $\sigma\in\{-1,1\}$, one has 
\begin{equation}\label{eq:transfer}
 \sup_{|h|\le1}\sigma\Delta_K^{(k)}(t^{mk}e^h)
 \ge\sigma\mathcal V_N(t)
 +O\left(t^{2mk}e^{-ct^2N^{\frac{2}{mk}}}\right),
\end{equation}
where $c>0$ is an absolute constant.
\end{lemma}

\begin{proof}
Let $S$ denote the supremum in~\eqref{eq:transfer}. Write
$w_N(u)=N^{\frac{1}{mk}}e^{-u^2N^{\frac{2}{mk}}}/\sqrt\pi$ for the normalized Gaussian
in~\eqref{eq:gaussian}.
The divisor bound gives $|\Delta_K^{(k)}(y)|\ll_{K,k}1+y^2$, and hence
$|S|\ll t^{2mk}$. Now Gaussian tail estimates give
\begin{equation*}
 \int_{|u|>t}|\Delta_K^{(k)}(t^{mk}e^{u/t})|w_N(u)\,du
 +t^{2mk}\int_{|u|>t}w_N(u)\,du
 \ll t^{2mk}e^{-ct^2N^{\frac{2}{mk}}}.
\end{equation*}
For the first integral, use
$-N^{\frac{2}{mk}}u^2+2u/t\le-\tfrac12N^{\frac{2}{mk}}u^2$ when $|u|>t$ and $t$ is
large. On $|u|\le t$, the signed error term is at most $S$.
Splitting~\eqref{eq:gaussian} at $|u|=t$ therefore gives
$\sigma\mathcal V_N(t)\le S+O(t^{2mk}e^{-ct^2N^{2/(mk)}})$.
\end{proof}

\subsection{An odd Gaussian kernel}
For $H>0$, let
\begin{equation}\label{eq:gauss-kernel}
 K_H(u)=H^2u e^{-H^2u^2/2}.
\end{equation}
The kernel is odd and has $L^1$ norm $2$. With our Fourier convention,
differentiation of the Gaussian transform gives
\begin{equation}\label{eq:gauss-transform}
 \widehat K_H(\xi)=-i\sqrt{2\pi}\frac{2\pi\xi}{H}
 e^{-(2\pi\xi)^2/(2H^2)}.
\end{equation}
We use this kernel to turn a sine sum into a cosine sum with
nonnegative coefficients.
\begin{lemma}\label{lem:odd}
Let $S(t)=\sum_nf(n)\sin(2\pi\lambda_nt)$, where $f(n)\ge0$, $\sum_nf(n)<\infty$, and $\lambda_n>0$. Then
\begin{equation}\label{eq:odd-convolution}
 \int_{\R}S(t+u)K_H(u)\,du
 =\sqrt{2\pi}\sum_nf(n)\frac{2\pi\lambda_n}{H}
 e^{-(2\pi\lambda_n)^2/(2H^2)}\cos(2\pi\lambda_nt),
\end{equation}
where the coefficients on the right are nonnegative. If $1\le2\pi\lambda_n/H\le e^{1/(mk)}$, the multiplier of $f(n)$
is bounded below by a positive constant depending only on $mk$.
\end{lemma}

\begin{proof}
Since $\|K_H\|_1\sum_nf(n)<\infty$, we may integrate term by term.
The kernel is odd, so the integral of
$\cos(2\pi\lambda_nu)K_H(u)$ is zero.
Also,
\begin{equation*}
 \int_{\R}\sin(2\pi\lambda_nu)K_H(u)\,du
 =-\Imag\widehat K_H(\lambda_n).
\end{equation*}
By the addition formula for sine and~\eqref{eq:gauss-transform},
then we have
\begin{equation*}
 \begin{aligned}
 \int_{\R}S(t+u)K_H(u)\,du
 &=-\sum_nf(n)\cos(2\pi\lambda_nt)
                  \Imag\widehat K_H(\lambda_n)\\
 &=\sqrt{2\pi}\sum_nf(n)\frac{2\pi\lambda_n}{H}
 e^{-(2\pi\lambda_n)^2/(2H^2)}\cos(2\pi\lambda_nt).
 \end{aligned}
\end{equation*}
This proves~\eqref{eq:odd-convolution}.
Finally, $v e^{-v^2/2}$ is positive and continuous on
$[1,e^{1/(mk)}]$, so it has a positive minimum there.
This gives the stated lower bound for the multiplier.
\end{proof}

\section{Proof of the main theorem}
\begin{proof}[Proof of Theorem~\ref{thm:main}]
For large $T$, we take $M=\lfloor\log T/(16\log2)\rfloor$ and 
$N=(\log T)^2$. Let $F_N(t)=\sum_nf_N(n)e^{2\pi i\lambda_nt}$.
Fix $0<\varepsilon<1/4$. By the divisor bound, we have
\begin{equation}\label{eq:Fzero}
 F_N(0)\ll_\varepsilon
 \sum_{n\ge1}n^{-\alpha+\varepsilon}e^{-c(n/N)^{2/(mk)}}
 \ll_\varepsilon N^{\eta+\varepsilon}.
\end{equation}
For $T/2\le t\le2T^2$ and sufficiently large $T$, we have
$2\le N\le(\log t)^3$, so Proposition~\ref{prop:voronoi} applies.
After division by $t^{(mk-1)/2}$, its error is
\begin{equation*}
 O_\varepsilon\left(t^{-1/10}N^{1/2+\varepsilon}\right)
 =O_\varepsilon\left(T^{-1/10}(\log T)^{1+2\varepsilon}\right),
\end{equation*}
uniformly on this interval.
We will also use $2^M\le T^{1/16}$.
Together with~\eqref{eq:Fzero}, this gives
\begin{equation}\label{eq:resonance-error-main}
 F_N(0)2^M\left(\frac{Y\log X}{X}\right)^{1/4}
 \ll_\varepsilon
 T^{-3/16}(\log T)^{2\eta+2\varepsilon+1/4}=o(1)
\end{equation}
when $X=T^2$ and $Y=T/2$ or $Y=T$.

\medskip\noindent\emph{Case 1:} Suppose that $kr_1\not\equiv1,5\pmod8$. Let $\sigma$ be the sign of
$\cos\beta$. Then $\sigma e^{i\beta}=e^{i\beta_0}$ for some
$\beta_0\in\{-\pi/4,0,\pi/4\}$.
We take $\M$ to be the indices of the largest $M$ positive values
of $a(n)n^{-\alpha}$. For each $n\in\M$, we have $n^\alpha/a(n)\le y_M$.
By~\eqref{eq:asymp_y} and $\kappa>0$, we have $y_M\ll M^\alpha$.
Also, the divisor bound gives $a(n)\ll n^{\alpha/4}$, so $ n^{3\alpha/4}\ll n^{\alpha}/{a(n)}\le y_M\ll M^\alpha$.
Thus $n\ll M^{4/3}$, and hence $n/N=o(1)$ uniformly on $\M$.
The exponential factor in $f_N(n)$ therefore tends uniformly to one
on $\M$. By Proposition~\ref{prop:largest}, we get
\begin{equation}\label{eq:selectedmass}
 \sum_{n\in\M}f_N(n)
 \sim S_K(M)\asymp(\log T)^\eta(\log_2T)^\kappa.
\end{equation}

Now we apply Lemma~\ref{lem:resonance} with $Y=T/2$ and $X=T^2$.
Its error tends to zero by~\eqref{eq:resonance-error-main}.
Since the right side of~\eqref{eq:selectedmass} tends to infinity,
we obtain a point $t\in[T/2,T^2]$ for which
$ \sigma\Real\left(e^{i\beta}F_N(t)\right)
 \gg(\log T)^\eta(\log_2T)^\kappa$.
We use this in~\eqref{eq:voronoi}. Then the error estimated above gives $ \sigma\mathcal V_N(t)
 \gg t^{(mk-1)/2}(\log T)^\eta(\log_2T)^\kappa$.
Since the error in Lemma~\ref{lem:transfer} is also negligible,
we can therefore choose $x\in[e^{-1}t^{mk},e t^{mk}]$ such that
\begin{equation*}
 \sigma\Delta_K^{(k)}(x)
 \gg t^{(mk-1)/2}(\log T)^\eta(\log_2T)^\kappa.
\end{equation*}
Since $x\asymp t^{mk}$, we have
$t^{(mk-1)/2}\asymp x^\eta$.
Also, $\log x\asymp\log T$ and $\log_2x\sim\log_2T$.
It follows that
\begin{equation*}
 \sigma\Delta_K^{(k)}(x)
 \gg(x\log x)^\eta(\log_2x)^\kappa.
\end{equation*}
These values of $x$ tend to infinity as $T\to\infty$.
The sign $\sigma$ is positive when $kr_1\equiv2,3,4\pmod8$
and negative when $kr_1\equiv0,6,7\pmod8$.
This proves part~(i).

\medskip\noindent\emph{Case 2:} Suppose that $kr_1\equiv1,5\pmod8$, and take $ S_N(t)=\sum_nf_N(n)\sin(2\pi\lambda_nt)$.
The cosine sum in~\eqref{eq:voronoi} equals $S_N(t)$ in class $1$
and $-S_N(t)$ in class $5$.
By Proposition~\ref{prop:band}, we can choose a set
$\M\subset[U,eU]$ of $M$ indices, all satisfying
$n\le M(\log M)^{C_2}$.
Thus, for $n\in\M$, we have $n/N = (\log T)^{-1}(\log_2 T)^{C_2} =o(1)$.
Again the exponential factor in $f_N(n)$ tends uniformly to one
on the selected set, and then Proposition~\ref{prop:band}  gives $sum_{n\in\M}f_N(n)\gg\mathcal A(T)$, where
\begin{equation}\label{eq:AT}
 \mathcal A(T)=(\log T)^\eta(\log_2T)^\kappa
               (\log_3T)^{-\alpha/2}.
\end{equation}
We now take $H=2\pi mk\left( U/{\disc^k}\right)^{1/(mk)}$.
For $n\in\M$, we have
$1\le2\pi\lambda_n/H\le e^{1/(mk)}$.
By Lemma~\ref{lem:odd}, convolution with $K_H$ gives
\begin{equation*}
 \int_{\R}S_N(t+u)K_H(u)\,du
 =\sum_ng_N(n)\cos(2\pi\lambda_nt),
\end{equation*}
where
\begin{equation*}
 g_N(n)=\sqrt{2\pi}f_N(n)\frac{2\pi\lambda_n}{H}
        e^{-(2\pi\lambda_n)^2/(2H^2)}\ge0.
\end{equation*}
The multiplier is bounded below on $\M$, so
$\sum_{n\in\M}g_N(n)\gg\mathcal A(T)$.
Also, $v e^{-v^2/2}\le e^{-1/2}$ for $v\ge0$, and hence $\sum_ng_N(n)\le\sqrt{2\pi/e}\,F_N(0)$.
We can therefore apply Lemma~\ref{lem:resonance} to these coefficients
with $\beta_0=0$, $Y=T$, and $X=T^2$.
By~\eqref{eq:resonance-error-main}, its error again tends to zero.
Since $\mathcal A(T)\to\infty$, we get a point $t\in[T,T^2]$
such that
\begin{equation*}
 \int_{\R}S_N(t+u)K_H(u)\,du\gg\mathcal A(T).
\end{equation*}
Now the interval $[U,eU]$ contains a positive integer, so $U\ge e^{-1}$.
Thus $H$ is bounded below by a positive constant depending only on
$K$ and $k$.
Since $|S_N(t+u)|\le F_N(0)$, the part with $|u|>T/2$ satisfies
\begin{equation*}
 \left|\int_{|u|>T/2}S_N(t+u)K_H(u)\,du\right|
 \le2F_N(0)e^{-H^2T^2/8}.
\end{equation*}
Again, the integral over $|u|\le T/2$ is
$\gg\mathcal A(T)$.
Using $\|K_H\|_1=2$, we get a point
$s\in[t-T/2,t+T/2]\subset[T/2,2T^2]$ such that
$|S_N(s)|\gg\mathcal A(T)$.
Now the equation~\eqref{eq:voronoi} and its error estimate imply $ |\mathcal V_N(s)|\gg s^{(mk-1)/2}\mathcal A(T)$.
We apply Lemma~\ref{lem:transfer} with $\sigma$ equal to the sign
of $\mathcal V_N(s)$.
As before, we can choose $x\in[e^{-1}s^{mk},e s^{mk}]$ for which $ |\Delta_K^{(k)}(x)|\gg x^\eta\mathcal A(T)$. Hence $\log x\asymp \log T$ proves part~(ii).
\end{proof}

\section{Applications}
\label{sec:applications}

We first give examples of the splitting exponent and then apply
Theorem~\ref{thm:main} to several counting problems.
As before, $k\ge1$ and $mk\ge2$. We define
\begin{equation*}
 A_{K,k}(x):=\sum_{n\le x}d_K^{(k)}(n)
 =M_{K,k}(x)+\Delta_K^{(k)}(x),
\end{equation*}
where $M_{K,k}(x)=\Res_{s=1}(\zeta_K(s)^kx^s/s)$.
For $k=1$, we have $M_{K,1}(x)=\rho_Kx$, with
$\rho_K=\Res_{s=1}\zeta_K(s)$. Put
\begin{equation*}
 W_{K,k}(x):=(x\log x)^\eta(\log_2x)^\kappa
 \begin{cases}
  1,&kr_1\not\equiv1,5\pmod8,\\
  (\log_3x)^{-\alpha/2},&kr_1\equiv1,5\pmod8.
 \end{cases}
\end{equation*}
Here $\alpha$, $\eta$ and $\kappa$ are given by
\eqref{eq:R-eta} and~\eqref{eq:splitting-exponent}.
If a count equals $cA_{K,k}(x)$ with $c>0$, its error is
$c\Delta_K^{(k)}(x)$, so the theorem gives the same sign.
Only the last application needs a further argument.

\subsection{Examples of number fields}

Only $\delta_\nu$ with $\nu>0$ contribute to $\kappa$.
If $K/\Q$ is Galois of degree $m$, the only nonzero density
among these is $\delta_m=1/m$, so
\begin{equation*}
 \kappa=\frac{mk+1}{2mk}
 \left(\frac{(mk)^{2mk/(mk+1)}}m-1\right).
\end{equation*}
Taking $K=\Q$ gives the classical $k$-divisor problem for $k\ge2$.
A totally imaginary field gives $\Omega_-$ for every $k$,
while a totally real field gives a bound depending on $mk$ modulo eight.
For $k=1$, real and imaginary quadratic fields give $\Omega_+$
and $\Omega_-$, respectively, with
\begin{equation}\label{eq:quadratic-scale}
 W_{K,1}(x)=(x\log x)^{1/4}
 (\log_2x)^{\frac34(2^{1/3}-1)}.
\end{equation}

For a non-Galois cubic field, we have $\delta_1=1/2$ and $\delta_3=1/6$.
Thus, for $k=1$, we get $\eta=1/3$ and $\kappa=(\sqrt3-1)/3$.
The bound is $\Omega_+$ with no factor involving $\log_3x$
when the field is totally real. With one real embedding, it is
unsigned and contains $(\log_3x)^{-1/3}$.

For a quintic field with Galois closure group $S_5$, we have
$\delta_1=3/8$, $\delta_2=1/6$, $\delta_3=1/12$ and $\delta_5=1/120$.
For $k=1$, this gives
\begin{equation*}
 \kappa=\frac35\left(
 \frac38+\frac{2^{5/3}}6+\frac{3^{5/3}}{12}
 +\frac{5^{5/3}}{120}-1\right).
\end{equation*}
The bound is $\Omega_+$ for $r_1=3$. For $r_1=1$ or $5$, it is
unsigned and contains $(\log_3x)^{-3/10}$.
A root of $X^5-5X^3-X^2+5X+1$ gives a totally real field
of this type~\cite[Section~5]{VZ}, in the exceptional class $r_1=5$.
More generally, the natural action of $S_m$ gives $R=m-1$ for $m\ge2$.

Now let $K_p=\Q(\sqrt[p]{2})$ for an odd prime $p$.
It has one real embedding, and its Galois closure group
$\operatorname{AGL}(1,p)$ acts on $\mathbb F_p$.
The identity fixes $p$ points, nonidentity translations fix none,
and every other element fixes one. Thus
$\delta_1=(p-2)/(p-1)$ and $\delta_p=1/(p(p-1))$.
For $k=1$, we get
\begin{equation*}
 \kappa=\frac{p+1}{2p(p-1)}
 \left(p^{(p-1)/(p+1)}-1\right).
\end{equation*}
The bound is unsigned and contains $(\log_3x)^{-(p+1)/(4p)}$.
Taking $k\equiv5\pmod8$ gives examples in the other exceptional class.

\subsection{Binary quadratic forms}

Let $D<0$ be a fundamental discriminant and put $K=\Q(\sqrt D)$.
Let $\mathcal S_D$ contain one representative from each proper class
of primitive positive-definite binary quadratic forms of discriminant $D$.
We define
\begin{equation*}
 r_D(n):=\sum_{Q\in\mathcal S_D}
 \#\{(u,v)\in\Z^2:Q(u,v)=n\},
\end{equation*}
with no coprimality condition on $u,v$.
If $w_D=|\mathcal O_K^\times|$, the form--ideal correspondence gives
$r_D(n)=w_Dd_K^{(1)}(n)$. Indeed, a form attached to an integral ideal $I$
represents $N_{K/\Q}(\xi)/N(I)$ for $\xi\in I\setminus\{0\}$,
and $\xi\mapsto(\xi)I^{-1}$ counts each ideal in the corresponding
class $w_D$ times.

Let $N_{D,k}(x)$ count ordered choices of $k$ forms in $\mathcal S_D$
and $k$ nonzero integer pairs whose represented values have product
at most $x$. By Dirichlet convolution, $N_{D,k}(x)=w_D^kA_{K,k}(x)$, so
\begin{equation*}
 N_{D,k}(x)-w_D^kM_{K,k}(x)=\Omega_-\!\left(W_{K,k}(x)\right).
\end{equation*}
For $N_Q(x)=\#\{(u,v)\in\Z^2:0<Q(u,v)\le x\}$,
the class number formula gives
\begin{equation*}
 \sum_{Q\in\mathcal S_D}
 \left(N_Q(x)-\frac{2\pi}{\sqrt{|D|}}x\right)
 =w_D\Delta_K^{(1)}(x).
\end{equation*}
Since $\mathcal S_D$ is finite, at least one fixed form satisfies
\begin{equation*}
 N_Q(x)-\frac{2\pi}{\sqrt{|D|}}x
 =\Omega_-\!\left(W_{K,1}(x)\right),
\end{equation*}
with $W_{K,1}$ as in~\eqref{eq:quadratic-scale}.
For class number one, this applies to the unique class.
Examples are $u^2+v^2$, $u^2+2v^2$ and $u^2+uv+v^2$, with main terms
$\pi x$, $\pi x/\sqrt2$ and $2\pi x/\sqrt3$.
Including the origin does not affect these bounds.

\subsection{Lattice points in norm-form regions}

By choosing generators in a fundamental domain for the units,
we can express ideal counts as counts of integer points.
We give two examples using the field data in~\cite{LMFDB}.

Take $K=\Q(\zeta_8)$, where $\zeta_8=e^{\pi i/4}$.
For $\mathbf a=(a,b,c,d)\in\Z^4$, put
$A(\mathbf a)=a^2+b^2+c^2+d^2$ and
$B(\mathbf a)=ab-ad+bc+cd$. Let $N_4(x)$ count the vectors with
\begin{equation*}
 0<A(\mathbf a)^2-2B(\mathbf a)^2\le x,
 \qquad 0\le3B(\mathbf a)<2A(\mathbf a).
\end{equation*}
\begin{corollary}\label{cor:quartic-norm}
We have $N_4(x)=8A_{\Q(\zeta_8),1}(x)$ and
\begin{equation}\label{eq:quartic-norm-bound}
 N_4(x)-\frac{\pi^2\log(1+\sqrt2)}2x
 =\Omega_-\!\left(
 (x\log x)^{3/8}(\log_2x)^{\frac58(4^{3/5}-1)}\right).
\end{equation}
\end{corollary}
\begin{proof}
We have
\begin{equation*}
 N_{K/\Q}(a+b\zeta_8+c\zeta_8^2+d\zeta_8^3)
 =A(\mathbf a)^2-2B(\mathbf a)^2.
\end{equation*}
The ring $\Z[\zeta_8]$ has class number one, with units
$\zeta_8^j(1+\sqrt2)^r$, where $0\le j<8$ and $r\in\Z$.
Also, the second inequality defining $N_4(x)$ is equivalent to
\begin{equation*}
 1\le\frac{A+\sqrt2B}{A-\sqrt2B}<(1+\sqrt2)^4.
\end{equation*}
Multiplication by $1+\sqrt2$ multiplies this ratio by $(1+\sqrt2)^4$.
Thus each nonzero integral ideal contributes exactly eight generators,
one for each root of unity, and $N_4(x)=8A_{K,1}(x)$.
Now apply Theorem~\ref{thm:main}, using
$\rho_K=\pi^2\log(1+\sqrt2)/16$.
\end{proof}

Next take $K=\Q(\theta)$, where $\theta=\sqrt[3]{2}$, and put
$\varepsilon=1+\theta+\theta^2$. We have
\begin{equation*}
 P(a,b,c):=N_{K/\Q}(a+b\theta+c\theta^2)
 =a^3+2b^3+4c^3-6abc.
\end{equation*}
Let $N_3(x)$ count the integer triples with
\begin{equation*}
 0<P(a,b,c)\le x,\qquad
 P(a,b,c)\le(a+b\theta+c\theta^2)^3<\varepsilon^3P(a,b,c).
\end{equation*}
\begin{corollary}\label{cor:cubic-norm}
We have $N_3(x)=A_{\Q(\sqrt[3]{2}),1}(x)$ and
\begin{equation*}
 N_3(x)-\frac{\pi\log\varepsilon}{3\sqrt3}x
 =\Omega\!\left((x\log x)^{1/3}
 (\log_2x)^{(\sqrt3-1)/3}(\log_3x)^{-1/3}\right).
\end{equation*}
\end{corollary}
\begin{proof}
The ring $\mathcal O_K=\Z[\theta]$ has class number one, with units
$\pm\varepsilon^r$ for $r\in\Z$.
Positive norm selects generators positive at the real embedding, and the second
condition selects one power of $\varepsilon$.
Thus each nonzero integral ideal contributes one generator,
so $N_3(x)=A_{K,1}(x)$. Theorem~\ref{thm:main} now applies with
$\rho_K=\pi\log\varepsilon/(3\sqrt3)$.
\end{proof}

\subsection{Invariant sublattices and integral matrices}

Let $m\ge2$ be a power of two, and define the signed cyclic shift
\begin{equation*}
 T(x_1,\ldots,x_m)=(-x_m,x_1,\ldots,x_{m-1}).
\end{equation*}
By identifying vectors with coefficient polynomials, $T$ becomes
multiplication by $X$ in $\Z[X]/(X^m+1)\simeq\Z[\zeta_{2m}]$.
Thus full-rank $T$-invariant sublattices, also called negacyclic lattices,
correspond to nonzero integral ideals of $K=\Q(\zeta_{2m})$,
and index equals norm. For related constructions, see
Fukshansky and Sun~\cite[Section~1]{FukshanskySun} and
Lyubashevsky and Micciancio~\cite[Section~3]{LyubashevskyMicciancio}.

Each sublattice has a unique basis matrix in column Hermite
normal form~\cite[Section~2.4]{Cohen}. We use upper-triangular matrices
with positive diagonal and $0\le B_{ij}<B_{ii}$ for $i<j$.
Let $\mathcal B_m(x)$ count these matrices with $\det B\le x$
and $B^{-1}TB$ integral. Then $\mathcal B_m(x)=A_{K,1}(x)$, so
\begin{equation*}
 \mathcal B_m(x)-\rho_Kx
 =\Omega_-\!\left((x\log x)^{(m-1)/(2m)}
 (\log_2x)^{\frac{m+1}{2m}(m^{(m-1)/(m+1)}-1)}\right).
\end{equation*}
Now let $L_{m,k}(x)$ count chains of full-rank $T$-invariant sublattices
\begin{equation*}
 \Z^m=\Lambda_0\supseteq\Lambda_1\supseteq\cdots\supseteq\Lambda_k,
 \quad \text{with} \quad [\Z^m:\Lambda_k]\le x.
\end{equation*}
For the corresponding ideals, $J_j=I_jI_{j-1}^{-1}$ is integral, and
conversely, $I_j=J_1\cdots J_j$. Thus $L_{m,k}(x)=A_{K,k}(x)$ and
\begin{equation*}
 L_{m,k}(x)-M_{K,k}(x)=\Omega_-\!\left(W_{K,k}(x)\right).
\end{equation*}
For $m=2$, the basis matrices are
\begin{equation*}
 B=\begin{pmatrix}dq&dr\\0&d\end{pmatrix},
 \qquad B^{-1}TB=\begin{pmatrix}-r&-s\\q&r\end{pmatrix},
\end{equation*}
where $d,q,s\ge1$, $0\le r<q$ and $qs-r^2=1$.
If $\mathcal Q(x)$ counts these integer quadruples with $d^2q\le x$,
then $\mathcal Q(x)=A_{\Q(i),1}(x)$ and
\begin{equation*}
 \mathcal Q(x)-\frac\pi4x=\Omega_-\!\left(W_{\Q(i),1}(x)\right).
\end{equation*}

\subsection{Abelian covers of curves}

Let $p\ge5$ be prime, and let $C_p$ be the smooth projective complex
curve with affine equation $y^2=x^p-1$. Its genus is $(p-1)/2$,
and $\tau_p(x,y)=(\zeta_px,y)$ is an automorphism.
Let $A_{C_p}(x)$ count connected unramified abelian Galois covers
$D\to C_p$ of degree at most $x$ to which $\tau_p$ lifts.
We count covers up to isomorphism over $C_p$, without counting
different lifts separately.

\begin{proposition}\label{prop:curve-covers}
For $K=\Q(\zeta_p)$, we have $A_{C_p}(x)=A_{K,1}(x)$ and
\begin{equation*}
 A_{C_p}(x)-\rho_Kx
 =\Omega_-\!\left((x\log x)^{\frac{p-2}{2(p-1)}}
 (\log_2x)^{\frac{p}{2(p-1)}((p-1)^{(p-2)/p}-1)}\right).
\end{equation*}
\end{proposition}
\begin{proof}
The differentials $x^{j-1}\,dx/y$, for $1\le j\le(p-1)/2$,
form a basis of holomorphic differentials, with eigenvalues
$\zeta_p^j$ under $\tau_p^*$.
Together with their conjugates, these give every nontrivial $p$-th
root of unity once. Thus $\tau_p$ has characteristic polynomial
$\Phi_p$ on $H_1(C_p,\Q)$.

It follows that $H_1(C_p,\Z)$ is a torsion-free rank-one module over
$\mathcal O_K=\Z[\zeta_p]$. By~\cite[Theorem~3.31]{MilneANT},
it is isomorphic to a nonzero fractional ideal $I$.
Every finite-index $\mathcal O_K$-submodule $J\subseteq I$ has the form
$J=\mathfrak a I$ for a unique integral ideal $\mathfrak a=JI^{-1}$.
Also, multiplicativity of ideal norms gives $[I:\mathfrak a I]=N\mathfrak a$.

By the classification of covers and the lifting criterion
in~\cite[Theorem~1.38 and Proposition~1.33]{Hatcher}, the required
covers correspond to the finite-index $\tau_p$-invariant subgroups
of $H_1(C_p,\Z)$. These are exactly its $\mathcal O_K$-submodules.
The covers are algebraic by the Riemann existence theorem,
and their degrees equal the subgroup indices.
Hence $A_{C_p}(x)=A_{K,1}(x)$.
Since $K/\Q$ is Galois of degree $p-1$ and totally imaginary,
Theorem~\ref{thm:main} gives the stated bound.
\end{proof}

For $p=5$, the curve has genus two and the bound is
\begin{equation*}
 A_{C_5}(x)-\rho_{\Q(\zeta_5)}x
 =\Omega_-\!\left((x\log x)^{3/8}
 (\log_2x)^{\frac58(4^{3/5}-1)}\right).
\end{equation*}
Also, Riemann--Hurwitz gives $g(D)=1+(p-3)\deg(D/C_p)/2$.
Thus $A_{C_p}(x)$ counts these covers with genus at most $1+(p-3)x/2$.

\subsection{Bravais colourings}

Let $n\ge3$, and suppose that $\Z[\zeta_n]$ has class number one.
A Bravais colouring partitions this planar module into cosets
of a similar submodule. We identify colourings only by permutations
of the colour names, with no further identification by rotations
or reflections.

By Baake and Grimm~\cite[Lemma~1 and Proposition~1]{BaakeGrimm},
these submodules are the nonzero principal ideals.
Thus the number $B_n(x)$ of colourings with at most $x$ colours is
$A_{K,1}(x)$ for $K=\Q(\zeta_n)$. Since $K$ is totally imaginary,
\begin{equation*}
 B_n(x)-\rho_Kx=\Omega_-\!\left(W_{K,1}(x)\right).
\end{equation*}
For the octagonal module $\Z[\zeta_8]$, the main term is
$\pi^2\log(1+\sqrt2)x/16$, and the bound has the same order
as in~\eqref{eq:quartic-norm-bound}.

\subsection{Spectral counting for tensor products}

Let $L=-\Delta$ be the positive Laplacian on $\R^2/(2\pi\Z)^2$,
restricted to mean-zero functions. Take $P_k=L^{\otimes k}$ on
the tensor product of $k$ copies of this space. Its eigenvalues are
\begin{equation*}
 \prod_{j=1}^k(a_j^2+b_j^2)
 \quad \text{for} \quad (a_j,b_j)\in\Z^2\setminus\{(0,0)\}.
\end{equation*}
If $N_{P_k}(x)$ counts them with multiplicity, the sum-of-two-squares
formula gives $N_{P_k}(x)=4^kA_{\Q(i),k}(x)$. Hence
\begin{equation*}
 N_{P_k}(x)-4^kM_{\Q(i),k}(x)
 =\Omega_-\!\left(W_{\Q(i),k}(x)\right).
\end{equation*}
For the connection with divisor sums, see Gramchev, Pilipovi\'c,
Rodino and Vindas~\cite{GramchevEtAl}.
For $k=2$, we get
\begin{equation*}
 N_{P_2}(x)-\pi^2x\log x-Cx
 =\Omega_-\!\left((x\log x)^{3/8}
 (\log_2x)^{\frac58(2^{11/5}-1)}\right),
\end{equation*}
where $C=\pi^2(2\gamma-1)+8\pi L'(1,\chi_{-4})$,
$\gamma$ is Euler's constant, and $\chi_{-4}$ is the nonprincipal
character modulo four.

\subsection{Representation growth of Heisenberg groups}

Let $H_K$ be the group of upper unitriangular $3\times3$ matrices
over $\mathcal O_K$. Let $\widetilde r_{K,k}(n)$ count the
$n$-dimensional irreducible complex representations of $H_K^k$,
up to tensoring with one-dimensional characters.
By Stasinski and Voll~\cite[Theorem B]{StasinskiVoll} and the
tensor-product description for direct products, we have
\begin{equation*}
 Z_{K,k}(s):=\sum_{n\ge1}\frac{\widetilde r_{K,k}(n)}{n^s}
 =\frac{\zeta_K(s-1)^k}{\zeta_K(s)^k},\qquad \Real s>2.
\end{equation*}
We write $Q_{K,k}(x)=\Res_{s=2}(Z_{K,k}(s)x^s/s)$ and
\begin{equation*}
 E_{K,k}(x):=\sum_{n\le x}\widetilde r_{K,k}(n)-Q_{K,k}(x).
\end{equation*}
The main term is $x^2$ times a polynomial in $\log x$ of degree $k-1$.
\begin{proposition}\label{prop:heisenberg}
For $mk\ge2$, we have
\begin{equation*}
 E_{K,k}(x)=\Omega\!\left(xW_{K,k}(x)\right),
\end{equation*}
with the same sign as in Theorem~\ref{thm:main} whenever it gives a signed bound.
\end{proposition}
\begin{proof}
Let $a(n)=d_K^{(k)}(n)$ and
\begin{equation*}
 F(x):=\sum_{n\le x}na(n)
 -\Res_{s=2}\left(\zeta_K(s-1)^k\frac{x^s}{s}\right).
\end{equation*}
By $\zeta_K(s-1)^k=\zeta_K(s)^kZ_{K,k}(s)$, for $0<\epsilon<\eta$,
\begin{equation*}
 F(x)=\sum_{d\le x}a(d)E_{K,k}(x/d)+O_\epsilon(x^{1+\epsilon}).
\end{equation*}
The error comes from the tail $-\sum_{d>x}a(d)Q_{K,k}(x/d)$,
using $a(d)\ll_\epsilon d^\epsilon$ and the form of $Q_{K,k}$.
Also, partial summation gives
\begin{equation*}
 \Delta_K^{(k)}(x)=\frac{F(x)}x+\int_1^x\frac{F(t)}{t^2}\,dt+O(1).
\end{equation*}
Fix a sign $\sigma\in\{1,-1\}$ given by the main theorem, and 
$W_{K,k}(x)=x^\eta\ell(x)$ for large $x$.
The factor $\ell$ is eventually increasing, and extend it to a positive
nondecreasing function on $[1,\infty)$.
If $\sigma E_{K,k}(y)\le h\,yW_{K,k}(y)$ for all sufficiently large $y$,
then positivity of $a(d)$ gives
\begin{equation*}
 \sigma F(x)\le h\,\zeta_K(1+\eta)^k xW_{K,k}(x)
 +O_{h,\epsilon}(x^{1+\epsilon}).
\end{equation*}
Terms with bounded $x/d$ are included in the error.
Since $\int_1^xW_{K,k}(t)\,dt/t\le W_{K,k}(x)/\eta$,
the partial summation formula gives
\begin{equation*}
 \sigma\Delta_K^{(k)}(x)
 \le h\,\zeta_K(1+\eta)^k(1+\eta^{-1})W_{K,k}(x)
 +o\!\left(W_{K,k}(x)\right).
\end{equation*}
If the claimed signed bound failed, this would hold for every $h>0$,
contrary to Theorem~\ref{thm:main}.
The same argument with absolute values gives the unsigned bound.
\end{proof}

For $K=\Q(\zeta_8)$ and $k=1$, we have
$Q_{K,1}(x)=\rho_Kx^2/(2\zeta_K(2))$, so
\begin{equation*}
 E_{K,1}(x)=\Omega_-\!\left(
 x^{11/8}(\log x)^{3/8}(\log_2x)^{\frac58(4^{3/5}-1)}\right).
\end{equation*}

\end{document}